\documentclass[11pt]{article}
\usepackage{amsmath,amssymb,amsthm}
\usepackage{float}
\usepackage[top=2cm,bottom=2cm,left=3cm,right=3cm,marginparwidth=1.75cm]{geometry}
\usepackage{graphicx}
\usepackage{xcolor}
\usepackage{authblk}
\usepackage{hyperref}
\usepackage{mathrsfs}
\usepackage{upgreek}
\usepackage{enumitem}
\usepackage{bbm}

\usepackage[square,numbers]{natbib}
\theoremstyle{plain}
\newtheorem{theorem}{Theorem}[section]
\newtheorem{lemma}[theorem]{Lemma}
\newtheorem{proposition}[theorem]{Proposition}
\newtheorem{proposition*}{Proposition}

\newtheorem{claim}{Claim}

\theoremstyle{definition}

\newtheorem{example}{Example}[section]
\newtheorem*{example*}{Example}

\theoremstyle{remark}

\numberwithin{equation}{section}

\title{Improved Survival Results for the One-Dimensional Renewal Contact Process}
\date{}
\author[1]{Gustavo O. de Carvalho}
\author[2]{Lucas R. de Lima}
\affil[1]{\small Department of Statistics, Institute of Mathematics, Statistics and Scientific Computing, University of Campinas--SP, Brazil. \textit{e-mail:} \texttt{godc@unicamp.br}}
\affil[2]{\small Department of Mathematics, Institute of Mathematics and Computer Science, University of S\~ao Paulo, S\~ao Carlos--SP, Brazil. \textit{e-mail:} \texttt{lrdelima@icmc.usp.br}}

\DeclareMathOperator{\supp}{\mathrm{supp}}

\begin{document}

\maketitle

\begin{abstract}
The renewal contact process is a non-Markovian variant of the classical contact process in which recoveries are governed by independent renewal processes with interarrival distribution $\mu$. We establish new sufficient conditions ensuring finiteness of the critical infection parameter $\lambda_c(\mu)$ for the one-dimensional model. In particular, we prove that $\lambda_c(\mu)<+\infty$ for every non-degenerate arithmetic interarrival distribution. Moreover, finiteness holds whenever the atomic component of the renewal measure is uniformly small on sufficiently short intervals. This criterion applies in particular to all non-atomic interarrival distributions, including singular continuous laws. The proof combines local estimates for renewal measures with a comparison to a regenerative oriented percolation model and a Peierls-type contour argument.
\end{abstract}


\medskip


\section{Introduction}

The contact process, introduced by Harris in \cite{Harris1974}, is one of the fundamental models in the theory of interacting particle systems. In its classical form, each vertex of $\mathbb Z^d$ represents an individual that may be either healthy or infected. Infected vertices recover at rate $1$, while healthy vertices become infected at rate proportional to the number of infected neighbors. The model is Markovian and admits the graphical construction introduced by Harris in \cite{Harris1978}, which plays a central role in the probabilistic analysis of interacting systems; see for instance \cite{Durrett1995,Liggett1985}.

A natural non-Markovian extension of the contact process was introduced by Fontes, Marchetti, Mountford and Vares in \cite{fontes2019}. In this model, called the \emph{renewal contact process}, recovery times are no longer governed by exponential clocks. Instead, each vertex $x\in\mathbb Z^d$ is equipped with an independent renewal process $\mathcal{R}_x$ with interarrival distribution $\mu$, whose marks represent recovery times at the vertex. Infection attempts between neighboring vertices still occur according to independent Poisson processes of rate $\lambda$.

The resulting dynamics retain the graphical character of the classical contact process, but lose the Markov property. Indeed, the future evolution of the process depends on the elapsed time since the last renewal mark at each vertex. This temporal dependence creates substantial additional difficulties, particularly in one dimension where infection paths necessarily revisit vertices and repeatedly interact with the same renewal processes.

Given an initial infected set $A\subseteq\mathbb Z^d$, let $(\xi_t^A)_{t\ge0}$ denote the renewal contact process with interarrival distribution $\mu$, infection parameter $\lambda$, and initial condition $\xi_0^A=A$. By attractiveness, survival starting from a single infected vertex implies survival from every nonempty finite initial configuration. Writing $\xi_t:=\xi_t^{\{0\}}$, one defines the critical infection parameter by
\[
\lambda_c(\mu):=
\inf\left\{
\lambda>0:
\mathbb P\bigl(|\xi_t|\neq0 \text{ for all } t\ge0\bigr)>0
\right\}.
\]
A central problem in the study of the renewal contact process is to determine under which conditions on $\mu$ the critical value $\lambda_c(\mu)$ is zero, positive, finite, or infinite.

The first results in this direction were obtained in \cite{fontes2019}, where it was shown that sufficiently heavy-tailed interarrival distributions imply $\lambda_c(\mu)=0$. Later, \cite{fontes2020,fontes2023} established broad sufficient conditions ensuring positivity of the critical parameter. In particular, \cite{fontes2023} proved that $\lambda_c(\mu)>0$ whenever the interarrival distribution has finite $\alpha$-moment for some $\alpha>1$.

The complementary question of whether $\lambda_c(\mu)$ is finite appears to be considerably more delicate, especially in dimension one. In higher dimensions, survival may be achieved through infection paths that avoid revisiting vertices, allowing relatively direct comparisons with supercritical oriented percolation; see \cite{hilario2022}. In contrast, one-dimensional infection paths necessarily reuse vertices, producing long-range temporal dependencies through the renewal structure.

Recently, Santos and Vares \cite{santos2024} obtained important progress on this problem by proving finiteness of the critical parameter for classes of interarrival distributions, including continuous distributions with bounded support and absolutely continuous distributions with decreasing hazard rate. Their approach combines geometric crossing constructions with domination arguments for dependent oriented percolation models.

The present work extends their results by following a different strategy. Rather than relying on regularity assumptions on the hazard rate or on bounded-support hypotheses, we exploit local properties of the renewal measure itself. More precisely, we show that sufficiently small concentration of the atomic component of the renewal measure on short intervals yields a supercritical regenerative oriented percolation structure.

Our first result establishes that every non-degenerate arithmetic interarrival distribution gives rise to a finite critical parameter. Moreover, we obtain a criterion based on local bounds for the atomic component of the renewal measure. Roughly speaking, if the renewal measure does not concentrate too much atomic mass inside sufficiently short intervals, then the associated renewal contact process survives for large infection rates.

An important feature of our approach is that it does not rely on absolute continuity assumptions. In particular, every non-atomic interarrival distribution satisfies our criterion. This includes singular continuous laws, such as Cantor-type distributions, which are not covered by previous methods based on hazard-rate regularity.

The proofs combine ideas from renewal theory, oriented percolation, and contour methods. We first establish quantitative local estimates for the renewal measure and then construct a dependent oriented percolation model possessing a regenerative structure along vertical columns. The resulting dependence is controlled through a Peierls-type contour argument together with a parity decomposition of crossed columns.

The paper is organized as follows. In Section~\ref{sec:main.results} we state the main results and discuss examples of interarrival distributions covered by our criteria. Section~\ref{sec:renewal.bounds} develops the local estimates for renewal measures that constitute the probabilistic core of the argument. In Section~\ref{sec:regen_perc} we introduce the regenerative oriented percolation model and establish the key percolation estimate. The proofs of the main results are given in Section~\ref{sec:proofs}. Finally, Appendix~\ref{sec:appendix} contains the contour-counting estimates used in the Peierls argument.

\section{Main results} \label{sec:main.results}

We consider a renewal process whose interarrival distribution is a probability measure $\mu$ supported on $[0,+\infty)$ with $\mu(\{0\})=0$. Its mean is
\[
m := \int_0^{+\infty} x\,\mu(dx)\in(0,+\infty],
\]
and in the infinite-mean case we adopt the convention $m^{-1}=0$. For $n\ge 0$, let $\mu^{*n}$ denote the $n$-fold convolution of $\mu$, with $\mu^{*0}=\delta_0$. The associated \emph{renewal measure} is
\[
U := \sum_{n\ge 0} \mu^{*n},
\]
which records the distribution of all renewal epochs of the process. We say that $\mu$ is \emph{arithmetic} if there exists $d>0$ such that $\supp(\mu)\subseteq d\mathbb{N}$ almost surely with $\mathbb{N}:=\{1,2,\dots\}$. Otherwise, $\mu$ is \emph{non-arithmetic}. (These distributions are also commonly called lattice and non-lattice, respectively.)

\begin{theorem}\label{thm:arithmetic}
Let $\mu$ be a non-degenerate arithmetic distribution.  
Then $\lambda_c(\mu) \;<\; +\infty$.
\end{theorem}
\medskip

While the condition of being arithmetic in Theorem \ref{thm:arithmetic} is technical and may possibly be improved, the non-degenerate condition is necessary. In fact, if $\mu(\{a\})=1$ for some $a>0$, then the infection almost surely dies out at instant $a$ and thus $\lambda_c(\mu)=+\infty$.

In what follows, we use the fact that every probability measure $\nu$ on $[0,+\infty)$ admits a decomposition $\nu=\nu_{\mathrm{at}}+\nu_{\mathrm{cont}}$, where $\nu_{\mathrm{at}}$ is purely atomic and $\nu_{\mathrm{cont}}$
has no atoms. Hence, 
\[
\nu_{\mathrm{at}}:= \sum_{x:\,\nu(\{x\})>0}\nu(\{x\})\delta_x
\]
and $\nu_{\mathrm{cont}}=\nu-\nu_{\mathrm{at}}$. The proof of the next theorem relies on an explicit Peierls-type estimate. For convenience we fix the constant
\[
\eta:=2^{-8}=0.00390625.
\]
A derivation of this value is given in Appendix~\ref{sec:appendix}.

\begin{theorem}\label{thm:bounded}
Let $\mu$ be an interarrival distribution and let $U = \sum_{n\ge 0} \mu^{*n}$
be its renewal measure. If there exist $\kappa>0$ such that
\[
   \sup_{a\ge 0} U_{\mathrm{at}}\big( (a,a+\kappa] \big) \;<\; \eta/2,
\]
then $\lambda_c(\mu) \;<\; +\infty$. In particular, the conclusion holds for every non-atomic interarrival distribution.
\end{theorem}

We now illustrate our results with some concrete classes of distributions.

\begin{example}[Mixtures of Non-Atomic and Purely Atomic Laws]
\label{ex:mixture-U2}
Let $p\in(0,1)$ and consider $\mu$ to be the law of a random variable
\[
   X = \begin{cases}
       Y, & \text{with probability } 1-p,\\[1mm]
       Z, & \text{with probability } p,
   \end{cases}
\]
with $Y$ non-atomic and $Z$ purely atomic strictly positive random variables. Then the interarrival distribution $\mu$  can be written as
\[
   \mu = (1-p)\,\mu_Y + p\,\mu_Z,
\]
with $\mu_Y$ and $\mu_Z$ the laws of $Y$ and $Z$, respectively. For each $n\ge1$, the $n$-fold convolution of $\mu$ expands as
\[
   \mu^{*n}
   = \sum_{k=0}^n \binom{n}{k} (1-p)^{\,n-k} p^{\,k}\,
     \mu_Y^{*(n-k)} * \mu_Z^{*k}.
\]
 By Lema \ref{lem:atomic-part-convolution}, $(\mu^{*n})_{\mathrm{at}} = p^n\mu_Z^{\ast n}$. Therefore, $U_{\mathrm{at}} = \sum_{n\ge1} p^n\, \mu_Z^{*n}$, and
\[U_{\mathrm{at}}\bigl((0,+\infty) \bigl)= \frac{p}{1-p}.\]

Hence, by Theorem \ref{thm:bounded}, $\lambda_c(\mu)<+\infty$ whenever  $p<\frac{\eta/2}{1+\eta/2}=\frac{1}{513}$.
\end{example}

\begin{example}[Singular continuous distribution with unbounded support]
Let $\nu$ be the classical Cantor distribution on $[0,1]$, and let
\[
\mu:=\sum_{n=1}^{+\infty} 2^{-n}\,\nu(\,\cdot-(n-1)).
\]
Equivalently, if $X\sim \nu$ and $G$ is an independent geometric random variable with $\mathbb P(G=k)=2^{-k}$ for all $k\in\mathbb{N}$. Then $\mu$ is the law of $X+G-1$.

The measure $\mu$ is singular continuous, since it is a countable convex combination of translates of the Cantor distribution. Moreover, $\supp(\mu)=\bigcup_{n\in\mathbb N_0}\bigl(\mathbf C+n\bigr)$, where $\mathbf C$ denotes the classical Cantor set. Since $\mu$ has no atoms, Theorem~\ref{thm:bounded} implies that $\lambda_c(\mu)<+\infty$.
\end{example}

\section{Local bounds for the renewal measure} \label{sec:renewal.bounds}

In this section we recall some basic facts from renewal theory and establish local bounds for the renewal measure that will be used in the construction of the regenerative percolation process. Let $(X_i)_{i\ge1}$ be a sequence of i.i.d.\ nonnegative random variables with common distribution $\mu$, and define the renewal epochs
\[
S_n:=\sum_{i=1}^n X_i,
\qquad n\ge1,
\]
with $S_0:=0$. Recall that the associated renewal measure is $U:=\sum_{n\ge0}\mu^{*n}$, equivalently,
\[U(B)=\sum_{n\ge0}\mathbb{P}(S_n\in B)\quad\text{ for every Borel set }B\subseteq[0,+\infty).\]

\bigskip

The next lemma shows that the atomic component of the renewal measure is completely determined by the atomic component of the interarrival distribution.

\begin{lemma}[Stability of the atomic part under convolution]\label{lem:atomic-part-convolution}
Let $\mu$ be an interarrival distribution. Then for every $n\ge0$ the atomic part of $\mu^{*n}$ equals the $n$-fold convolution of the atomic part of $\mu$, i.e.
\[
\bigl(\mu^{*n}\bigr)_{\mathrm{at}} = \mu_{\mathrm{at}}^{*n}.
\]
Consequently, the atomic part of the renewal measure $U=\sum_{n\ge0}\mu^{*n}$ is
\[
U_{\mathrm{at}}=\sum_{n\ge0}\mu_{\mathrm{at}}^{*n}.
\]
\end{lemma}

\begin{proof}
We prove $(\mu^{*n})_{\mathrm{at}}=\mu_{\mathrm{at}}^{*n}$ by induction on $n$. The base case $n=0$ is trivial since $\mu^{*0}=\delta_0$ and
$\mu_{\mathrm{at}}^{*0}=\delta_0$. Assume the identity holds for some $n\ge0$. Write
\[
\mu^{*(n+1)} = \mu^{*n} * \mu
= (\mu^{*n}_{\mathrm{cont}}+\mu^{*n}_{\mathrm{at}}) * (\mu_{\mathrm{cont}}+\mu_{\mathrm{at}}),
\]
and expand the convolution into four terms. We claim that every term that contains at least one continuous factor is non-atomic. Indeed, if $\nu$ is
any probability measure with $\nu(\{x\})=0$ for all $x$ (i.e. $\nu$ has no atoms), then for any probability measure $\rho$ and any $z\in\mathbb R$
\[
(\rho*\nu)(\{z\})
= \int \rho(\{z-y\})\,\nu(dy)
= \sum_{a\in\mathrm{supp}(\rho_{\mathrm{at}})} \rho(\{a\})\,\nu(\{z-a\}) = 0,
\]
because $\nu(\{z-a\})=0$ for every singleton. Thus convolution with a non-atomic measure produces a non-atomic measure.

Applying this observation to the four terms in the expansion, the only term that can carry atoms is the convolution of the two atomic factors, $\mu^{*n}_{\mathrm{at}}*\mu_{\mathrm{at}}$. By the induction hypothesis,
\[
\bigl(\mu_{\mathrm{at}}^{*n} * \mu_{\mathrm{at}}\bigr)_{\mathrm{at}}
= \mu_{\mathrm{at}}^{*(n+1)}.
\]
This completes the induction.

By Tonelli's theorem, we may interchange the order of summation of the nonnegative atomic masses; hence
\[
U_{\mathrm{at}}=\sum_{n\ge0}\bigl(\mu^{*n}\bigr)_{\mathrm{at}}
=\sum_{n\ge0}\mu_{\mathrm{at}}^{*n}.
\]
as claimed.
\end{proof}

We now show that arithmetic interarrival distributions yield a uniform bound on the local masses of the renewal measure. In the following $\mathbb{N}_0:=\mathbb{N}\cup\{0\}$.

\begin{lemma}[Uniform bound for arithmetic renewal masses]\label{lem:discrete-U-atom-bound}
Let $\mu$ be a non-degenerate arithmetic interarrival distribution with support on $d\mathbb{N}$ for a $d>0$. Then $U=\sum_{n\in\mathbb{N}_0}\mu^{\ast n}$ is supported on $d\mathbb{N}_0$ with
\[
\sup_{k\in\mathbb{N}} U(\{d \cdot k \})< 1.
\]
\end{lemma}

\begin{proof}
Since $\mu$ is a non-degenerate arithmetic measure, $\supp(\mu) \subseteq d\mathbb{N}$ has at least two elements. Let $m:=d\cdot \sum_{k\in \mathbb{N}}k\cdot\mu(\{d\cdot k\})$. By the discrete renewal theorem (see Feller \cite[XI(1.9)]{feller1991} and recall the notation $m^{-1}=0$ when $m=+\infty$), we have
\[
    U\big(\{d \cdot k \}\big) \longrightarrow \frac{d}{m}
\quad\text{as }k\uparrow+\infty.
\]
Since $m>d$, the limit lies in $[0,1)$.

We now show that $U(\{d \cdot k \})<1$ for each fixed $k$. Fix $k\ge1$. Because $\mu$ is not concentrated on a single point, there exist $N\in \mathbb{N}$ and one finite sequence of integers
\[
(x_1,\dots,x_N)\in \supp(\mu)^N
\]
such that along the corresponding path of partial sums
\[
s_j := x_1+\cdots+x_j,\qquad j=1,\dots,N,
\]
we have $s_j\neq d \cdot k $ for all $j=1,\dots,N$, and moreover $s_N>d \cdot k $. Such a sequence exists because we can always choose $s_1\neq d \cdot k $ and, since $\supp(\mu)$ has at least two elements, if $s_{j-1}<d \cdot k $, then there exists at least one option to have $s_j\neq d \cdot k $.

Let $A_k$ be the event that the first $N$ increments of $(X_n)$ equal this particular sequence:
\[
A_k := \{X_1=x_1,\dots,X_N=x_N\}.
\]
Then $\mathbb{P}(A_k)>0$, since $\mu$ assigns positive mass to each $x_j$. On $A_k$ we have $S_j\neq d \cdot k$ for $1\le j\le N$ and $S_N>d \cdot k $, and since the sequence $(S_n)_{n\in\mathbb{N}}$ is strictly increasing, the process can never hit the level $d \cdot k$ after time $N$. Thus
\[
A_k \subseteq \{S_n\neq d \cdot k  \text{ for all }n\ge0\},
\]
and hence
\[
\mathbb{P}(S_n = d \cdot k  \text{ for some }n\in\mathbb{N})
\le 1 - \mathbb{P}(A_k) < 1.
\]
Therefore $U(\{d \cdot k \})<1$ for every fixed $k$.

Now define
\[
c_1 := \sup_{k\ge K} U(\{d \cdot k \}),
\]
where $K$ is chosen large enough so that
\[
U(\{d \cdot k \}) \le \frac{d}{m} + \varepsilon
\quad\text{for all } k\ge K,
\]
with some $\varepsilon>0$ small enough that
\[
\frac{d}{m} + \varepsilon < 1.
\]
Such $K$ and $\varepsilon$ exist by the convergence $U(\{d \cdot k \})\to d/m$. Then $c_1<1$.

On the finite set $\{0,1,\dots,K-1\}$, we have $U(\{d \cdot k \})<1$ for each $k$, as shown above. Hence
\[
c_0 := \max_{0\le k<K-1} U(\{d \cdot k \}) < 1.
\]
Finally, set $c := \max\{c_0,c_1\} < 1$. Then
\[
U(\{d \cdot k \}) \le c \quad\text{for all }k\ge0.
\]
\end{proof}

The next step is to control the continuous component of the renewal measure.

\begin{lemma}[Uniform smallness of the continuous part]\label{lem:uniform-small-intervals-continuous-part}
Let $U$ be the renewal measure associated to the interarrival distribution $\mu$ and write the Lebesgue decomposition $U = U_{\mathrm{cont}} + U_{\mathrm{at}}$. Then for every $\varepsilon>0$ there exists $\delta>0$ such that
\[
\sup_{a\ge 0} U_{\mathrm{cont}}\bigl((a,a+\delta]\bigr) < \varepsilon.
\]
\end{lemma}

\begin{proof}
If $\mu$ is purely atomic, then $U_{\mathrm{cont}}\equiv 0$ and the claim is trivial. Hence suppose that $\mu$ is non-atomic. Fix $\varepsilon>0$. The Renewal Theorem (see \cite[XI(1.9)]{feller1991}) implies that, for every fixed $h>0$,
\[
U\bigl((x,x+h]\bigr) \longrightarrow \frac{h}{m}
\qquad\text{as } x\to+\infty,
\]
where the right-hand side is interpreted as $0$ if $m=+\infty$. Choose $\delta'>0$ with $\delta'/m<\varepsilon/4$ and set $h=\delta'$. By the convergence above there exists $L>0$ such that for all $x\ge L$
\[
U\bigl((x,x+\delta']\bigr) < \frac{\delta'}{m} + \frac{\varepsilon}{4}
< \frac{\varepsilon}{2}.
\]
Since $U_{\mathrm{cont}}\le U$ as measures, it follows that for all $x\ge L$
\[
U_{\mathrm{cont}}\bigl((x,x+\delta']\bigr) \le
U\bigl((x,x+\delta']\bigr) < \frac{\varepsilon}{2}.
\]

It remains to control $U_{\mathrm{cont}}((a,a+\delta])$ for $a\in[0,L]$. Define the function
\[
F_{\mathrm{cont}}(t):=U_{\mathrm{cont}}((0,t])\qquad(t\ge0).
\]
Because $U_{\mathrm{cont}}$ has no atoms, $F_{\mathrm{cont}}$ is continuous. Moreover $F_{\mathrm{cont}}$ is finite on compact intervals, hence continuous on the compact interval $[0,L+\delta']$ and therefore uniformly continuous by Heine–Cantor theorem. Consequently there exists $\delta''\in(0,\delta']$ such that for every
$a\in[0,L]$ and every $\delta\in(0,\delta'']$,
\[
U_{\mathrm{cont}}((a,a+\delta]) = F_{\mathrm{cont}}(a+\delta)-F_{\mathrm{cont}}(a)
< \frac{\varepsilon}{2}.
\]

Combining the bounds for $a\in[0,L]$ and for $a\ge L$ we obtain
\[
\sup_{a\ge0} U_{\mathrm{cont}}((a,a+\delta]) \le \frac{\varepsilon}{2} < \varepsilon,
\]
which proves the lemma.
\end{proof}

\medskip
Combining the previous two lemmas allows us to transfer local bounds from the atomic component to the full renewal measure.

\begin{proposition}[Transfer of local control from the atomic part]
\label{prop:uniform-local-U_at}
Let $\mu$ be an interarrival distribution $\mu$ and consider $U$ to be its corresponding renewal measure. Suppose that there exist $\kappa>0$ and $\varepsilon>0$ such that 
\begin{equation}\label{eq:U_at.bound}
\sup_{a\ge 0} U_{\mathrm{at}}\big((a,a+\kappa]\big) < \varepsilon.
\end{equation}
Then there exists $\nu \in(0,\kappa]$ such that
\[
\sup_{a\ge 0} U\big((a,a+\nu]\big) < \varepsilon
\qquad\text{for all } \ a\ge 0.
\]
\end{proposition}

\medskip

\begin{proof}
Fix $\varepsilon>0$ and assume $\kappa>0$ satisfies \eqref{eq:U_at.bound}. Then for any $0<\nu\le \kappa$ and any $a\ge 0$,
\[
U_{\mathrm{at}}((a,a+\nu]) \le \sup_{a\ge 0} U_{\mathrm{at}}((a,a+\kappa]).
\]

Lemma \ref{lem:uniform-small-intervals-continuous-part} ensures the existence of $\nu\in (0,\kappa]$ such that \[\sup_{a \ge 0}U_{\mathrm{cont}}\big( (a,a+\nu] \big)< \varepsilon-\sup_{a \ge 0}U_{\mathrm{at}}((a,a+\kappa]).\] 

Hence, for such $\nu \in (0,\kappa]$, we have
\[
\sup_{a\ge 0}U((a,a+\nu]) 
\le \sup_{a \ge 0}U_{\mathrm{cont}}((a,a+\nu]) + \sup_{a \ge 0}U_{\mathrm{at}}((a,a+\kappa])<\varepsilon.
\]
\end{proof}

\section{Regenerative oriented percolation}\label{sec:regen_perc}

In this section we introduce a class of oriented percolation models with a regenerative structure along vertical columns. These models arise naturally from the graphical representation of the renewal contact process and will provide the main comparison tool in the proofs of the survival results.

Consider the oriented graph $\overrightarrow{\mathbb{L}}=(\mathbb{V},\mathbb{E})$, where the vertex set is given by
\begin{equation}\label{eq:def_v}
\mathbb{V}=\{(x,y)\in \mathbb{Z}\times \mathbb{N}_0:\text{$-y\leq x\leq y$ and $x+y$ is even}\}
\end{equation}
and the edge set is given by
\begin{equation}\label{eq:def_e}
\mathbb{E}=\{(v,v')\in\mathbb{V}^2:\text{$v'=v+(-1,1)$ or $v'=v+(1,1)$}\}.
\end{equation}

A bond percolation configuration on $\overrightarrow{\mathbb L}$ is given by a family of random variables $(X_e)_{e\in\mathbb E}$ taking values in $\{0,1\}$, where $X_e=1$ indicates that the edge $e$ is open. Otherwise, it is closed when $X_e=0$.

As usual, for $v,v'\in \mathbb{V}$ we write $v\to v'$ to denote the event in which there exists a sequence $v=v_0,v_1,\dots,v_k=v'$ such that $e_n=(v_{n-1},v_n)\in\mathbb{E}$ and $X_{e_n}=1$ for all $n\in\{1,...,k\}$. We also write, for  $v\in \mathbb{V}$,
\[\mathcal{C}_v:=\{v'\in\mathbb{V}:v\to v'\}\]
as the cluster starting in $v$. Moreover, we also denote
$\{|\mathcal{C}_0|=+\infty\}$
as the percolation event.

For an edge $e=((x,y),(x',y'))\in\mathbb E$, define its column by
\[
\mathrm{col}(e):=\min\{x,x'\}.
\]
A family of edges is called \textit{colinear} if all of them have the same column. Let $p\in(0,1)$ and consider any bond percolation model on $\overrightarrow{\mathbb{L}}$ that satisfies:
\begin{itemize}
    \item[(I)] $\mathbb{P}(X_{e_i}=0 \text{ for all } i\in\{1,\dots,k\})\le (1-p)^k$ whenever $\{e_1,\dots,e_k\}\subseteq\mathbb{E}$ are distinct and colinear;
    \item[(II)] $X_{e_1},...,X_{e_k}$ are independent whenever $|\mathrm{col}(e_i)-\mathrm{col}(e_j)|\ge 2$ for every $i,j\in\{1,...,k\}$ with $i\neq j$.
\end{itemize}

Note that usual independent site or bond percolation models on $\overrightarrow{\mathbb{L}}$ with parameter $p$ also satisfy the properties above. More importantly for our purposes, this class also includes models with some type of regenerative property given closed edges. In other words, we are interested in the cases where
\begin{equation}\label{eq:regen}
\mathbb{P}(X_{e_k}=0 \mid X_{e_1}=x_1,\dots,X_{e_{k-2}}=x_{k-2},X_{e_{k-1}}=0)\le 1-p
\end{equation}
for distinct colinear $\{e_1,\dots,e_k\}$ and any $(x_1,\dots,x_{k-2})\in\{0,1\}^{k-2}$ as long as the edges are indexed according to their height, with $e_k$ being the highest edge and $e_1$ the lowest (i.e., $e_i=(e_{i,1},e_{i,2})$ for $i=1,...,k$ with $e_{1,k}>e_{1,k-1}>...>e_{1,1}$). 
It is easy to see that this regenerative property implies (II). If we change from $X_{e_{k-1}}=0$ to $X_{e_{k-1}}=1$ in \eqref{eq:regen}, the inequality may no longer hold, which is in some sense analogous to what happens in a renewal process: independence (by regeneration) holds only given that the last known moment has a renewal mark; the event in which the last known moment does not have a mark carries some information of the current interarrival time.

The proof of Proposition~\ref{prop:percol} follows a Peierls-type contour argument. Finite clusters give rise to dual contours crossing the wedge, while assumptions {\rm (I)}--{\rm (II)} allow us to control the probability of such contours through a parity decomposition of the crossed columns. Let $\eta={1}/{2^8}$, as defined in Section \ref{sec:main.results}.

\begin{proposition}\label{prop:percol}
    For any percolation model satisfying $\mathrm{(I)}$ and $\mathrm{(II)}$ with $p>1-\eta$, we have
    \[
    \mathbb{P}(|\mathcal{C}_0|=+\infty)>0.
    \]
\end{proposition}

\begin{proof}

The proof follows from some modifications of the classical Peierls contour argument applied in the context of percolation (see, for example, Grimmett \cite{grimmett1999}).

We will construct another graph $\overrightarrow{\mathbb{L}}_*$, representing the dual graph of $\overrightarrow{\mathbb{L}}$. Let
\[
\mathbb{V}_*:=\{(x,y-1):(x,y)\in\mathbb{V}^2\}\setminus\{(0,-1)\}.
\]

Consider the following sets of edges
\[
\begin{cases}
\mathbb{E}_*^{\nearrow}:=\{(x,y)\in\mathbb{V}_*^{2}:x+(1,1)=y\},\\
\mathbb{E}_*^{\searrow}:=\{(x,y)\in\mathbb{V}_*^{2}:x+(1,-1)=y\},\\
\mathbb{E}_*^{\swarrow}:=\{(x,y)\in\mathbb{V}_*^{2}:x+(-1,-1)=y\},\\
\mathbb{E}_*^{\nwarrow}:=\{(x,y)\in\mathbb{V}_*^{2}:x+(-1,1)=y\},
\end{cases}
\]
and define
\[\mathbb{E}_*:=\mathbb{E}_*^{\nearrow}\cup \mathbb{E}_*^{\searrow} \cup \mathbb{E}_*^{\swarrow}\cup \mathbb{E}_*^{\nwarrow}.\]

We construct the usual bijective correspondence between $\mathbb{E}$ and $\mathbb{E}_*^{\nearrow}\cup \mathbb{E}_*^{\searrow}$. More specifically, to each edge $e=(v,v')\in\mathbb{E}$ we associate the unique dual edge $e_*=(v_*,v'_*)\in \mathbb{E}_*^{\nearrow}\cup \mathbb{E}_*^{\searrow}$ that is crossed by $e$, that is: if $v'=v+(1,1)$, then $e_*=(v+(1,0),v+(0,1))\in \mathbb{E}_*^{\searrow}$; if $v'=v+(-1,1)$, then $e_*=(v+(-1,0),v+(0,1))\in \mathbb{E}_*^{\nearrow}$. Moreover, we couple the primal and dual models so that $\{e\text{ open}\}\iff \{e_*\text{ closed}\}$. All the edges in $\mathbb{E}_*^{\swarrow}\cup \mathbb{E}_*^{\nwarrow}$ are considered open in every situation. 

Define the left and right sides of the dual wedge by
\[
\mathcal L:=\{(-y-1,y):y\ge0\},
\qquad
\mathcal R:=\{(y+1,y):y\ge0\}.
\]

\begin{claim}[Dual contour criterion]\label{claim:dual-contour}
If $|\mathcal{C}_0|<+\infty$, then there exists a self-avoiding open path in the dual graph $\overrightarrow{\mathbb{L}}_*$ connecting $\mathcal L$ to $\mathcal R$, the left and right sides of the wedge, and separating $\mathcal{C}_0$ from infinity.
\end{claim}
\begin{proof}\renewcommand{\qedsymbol}{$\blacksquare$}
Assume that $|\mathcal{C}_0|<+\infty$. Consider the set of dual edges crossing primal edges having exactly one endpoint in $\mathcal{C}_0$. By the primal--dual coupling, every such dual edge is open.

These dual edges form the external boundary of $\mathcal{C}_0$ in the dual graph. Since the primal graph is planar and $\mathcal{C}_0$ is finite, this boundary separates $\mathcal{C}_0$ from infinity. Moreover, because $\mathcal{C}_0$ is contained in the wedge, the boundary must connect the two sides of the wedge.

Finally, removing loops from the boundary path yields a self-avoiding open dual path with the same endpoints.
\end{proof}

We now formalise the contours that appear in Claim~\ref{claim:dual-contour}.
For $n\ge2$, let
\[\Gamma_n=\left\{{\begin{array}{c}
    \gamma=(\gamma_1,\dots,\gamma_{n+1}) \text{ self-avoiding dual paths with }\;(\gamma_i,\gamma_{i+1})\in\mathbb E_*\\
\;\text{for all }1\le i<n+1 \;\text{ such that }\;\gamma_1\in \mathcal L \;\text{ and }\;\gamma_{n+1}\in \mathcal R\\
    
\end{array}}\right\}.\]
A path $\gamma\in\Gamma_n$ is called \emph{open} if every dual edge belonging to it is open in the coupled model. 
By Claim~\ref{claim:dual-contour},
\[
\{|\mathcal{C}_0|<+\infty\} \;\subseteq\; \bigcup_{n\ge 2}\,\bigl\{ \text{there exists an open }\gamma\in\Gamma_n \bigr\},
\]
and in fact the two events are equal (the reverse inclusion is obvious because an open dual contour prevents infinite expansion of $\mathcal{C}_0$). 
An illustration of a contour is given in Figure~\ref{fig:periels}. Note that edges in $\mathbb{E}_*^\nearrow$, by construction, prevent $\mathcal{C}_0$ from expanding in the direction $(-1,1)$, while edges in $\mathbb{E}_*^\searrow$ prevent the expansion of $\mathcal{C}_0$ in the direction $(1,1)$. Moreover, since we are considering percolation on the oriented graph $\overrightarrow{\mathbb{L}}$, $\mathcal{C}_0$ cannot expand in the directions $(-1,-1)$ or $(-1,1)$; therefore, edges in $\mathbb{E}_*^\swarrow \cup \mathbb{E}_*^\nwarrow$ may be considered open in every situation.

\begin{figure}[h]
\centering
\includegraphics[width=10cm]{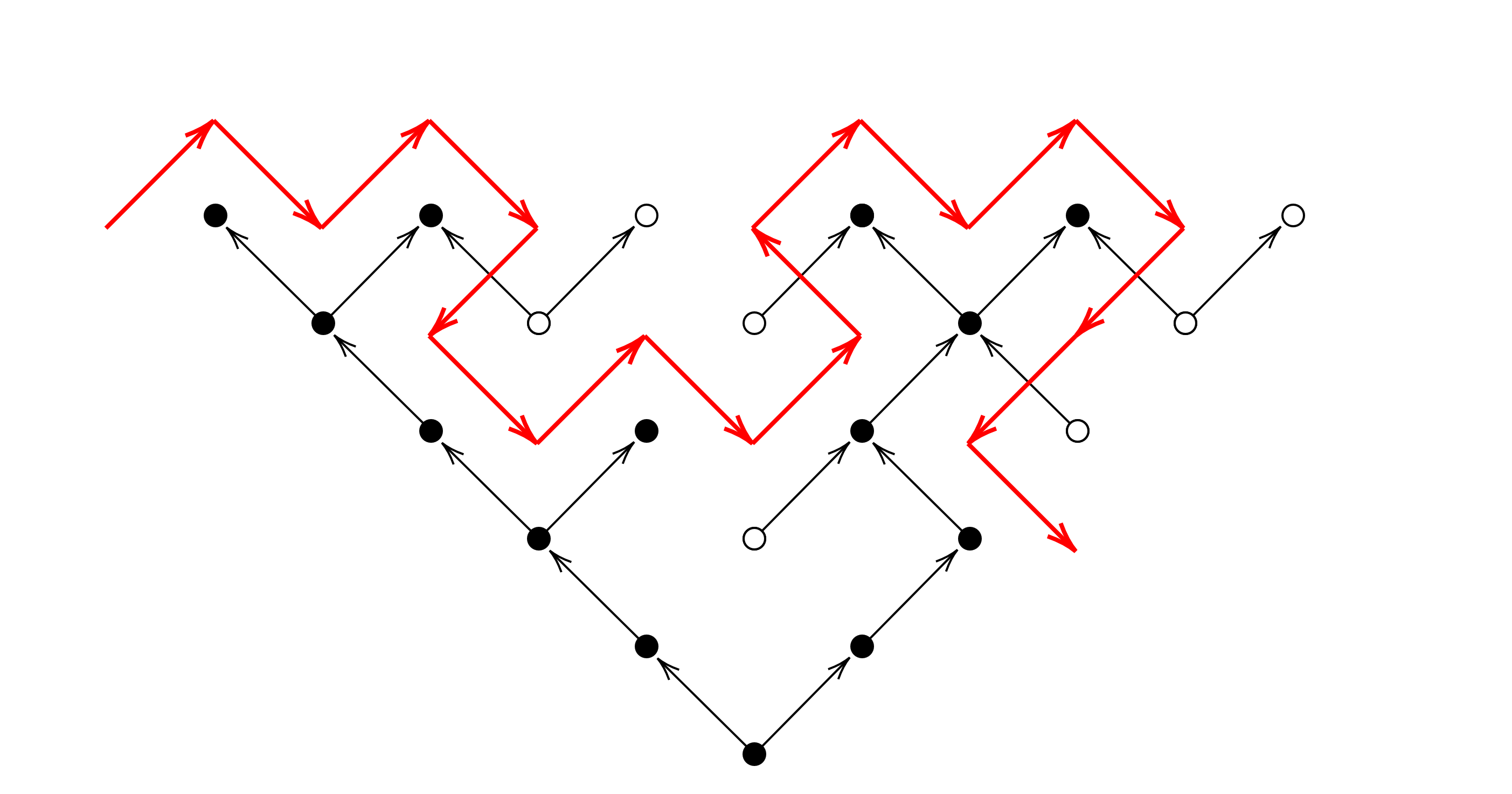}
\caption{Illustration of a contour. Black balls are vertices in $\mathcal{C}_0$ and white balls are the remaining ones. The red arrows represent the edges of an open path $\gamma\in \Gamma_{17}$.}\label{fig:periels}
\end{figure}

Therefore,
\begin{equation}\label{eq:p_contorno}
\mathbb{P}(|\mathcal{C}_0|< +\infty)
\le
\sum_{n\ge 2}\sum_{\gamma\in\Gamma_n} \mathbb{P}(\text{$\gamma$ open}).
\end{equation}

By Proposition~\ref{prop:peierls},we have
\begin{equation}
    |\Gamma_n|\le (n-1) 3^{n-2}.
\end{equation}

Fix $\gamma=(\gamma_1,\dots,\gamma_{n+1})\in \Gamma_n$. Note that the set
\[
\mathcal{A}_\gamma:=\{(\gamma_i,\gamma_{i+1})\in\mathbb{E}_*^{\nearrow}\cup\mathbb{E}_*^{\searrow}: i\in\{1,\dots,n\}\}.
\]

To exploit the independence assumption {\rm (II)}, we partition the dual edges according to the parity of their associated columns, obtaining the decomposition
\[
\mathcal{A}_\gamma=\mathcal{E}_\gamma\cup\mathcal{O}_\gamma,
\]
where
\[
\mathcal{E}_\gamma:=\{(\gamma_i,\gamma_{i+1})\in\mathbb{E}_*^{\nearrow}\cup\mathbb{E}_*^{\searrow}: i\in\{1,\dots,n\} \text{ and $\mathrm{col}\bigl((\gamma_i,\gamma_{i-1})\bigr)$ is even}\},
\]
\[
\mathcal{O}_\gamma:=\{(\gamma_i,\gamma_{i+1})\in\mathbb{E}_*^{\nearrow}\cup\mathbb{E}_*^{\searrow}: i\in\{1,\dots,n\} \text{ and $\mathrm{col}\bigl((\gamma_i,\gamma_{i-1})\bigr)$ is odd}\}.
\]

It is easy to note that, since $\gamma$ must travel from the left to the right side (see Figure~\ref{fig:periels}), then $\gamma$ has more edges in $\mathbb{E}_*^{\nearrow}$ than in $\mathbb{E}_*^{\swarrow}$ and more edges in $\mathbb{E}_*^{\searrow}$ than in $\mathbb{E}_*^{\nwarrow}$.
Thus $|\mathcal{A}_\gamma|\ge n/2$, and we have either $|\mathcal{E}_\gamma|\ge n/4$ or $|\mathcal{O}_\gamma|\ge n/4$. Without loss of generality, assume that $|\mathcal{E}_\gamma|\ge n/4$ and write
\begin{equation}
\mathcal{E}_\gamma=\bigcup_{k\in 2\mathbb{Z}} \mathcal{E}^{(k)}_\gamma,
\quad
\mathcal{E}^{(k)}_\gamma:=\{(\gamma_i,\gamma_{i+1})\in\mathbb{E}_*^{\nearrow}\cup\mathbb{E}_*^{\searrow}: i\in\{1,\dots,n\} \text{ and $\mathrm{col}\bigl((\gamma_i,\gamma_{i-1})\bigr)=k$}\}.
\end{equation}

Note that for $i,j\in\{1,\dots,n\}$ and $k,k'\in 2\mathbb{Z}$ with $k\neq k'$, we have by property (II) that $X_{(\gamma_i,\gamma_{i+1})}$ is independent of $X_{(\gamma_j,\gamma_{j+1})}$ whenever $(\gamma_i,\gamma_{i+1})\in \mathcal{E}^{(k)}_\gamma$ and $(\gamma_j,\gamma_{j+1})\in \mathcal{E}^{(k')}_\gamma$, since $|k-k'|\ge 2$. Applying the same idea for every edge in $\mathcal{E}_\gamma$ at once, we obtain

\begin{equation}
\mathbb{P}(\text{$\gamma$ open})
\le
\prod_{k\in 2\mathbb{Z}:\,\mathcal{E}^{(k)}_\gamma\neq\emptyset}
\mathbb{P}(\text{every edge in $\mathcal{E}^{(k)}_\gamma$ is open}).
\end{equation}

By construction, any edges $e^*_1,\dots,e^*_j\in \mathcal{E}^{(k)}_\gamma$ have corresponding primal edges $e_1,\dots,e_j$ that are colinear. By property (I) of the percolation model,
\begin{equation}\label{eq:ev_edge_open}
P(\text{every edge in $\mathcal{E}^{(k)}_\gamma$ is open})
\le
(1-p)^{|\mathcal{E}^{(k)}_\gamma|}.
\end{equation}

By displays \eqref{eq:p_contorno}-\eqref{eq:ev_edge_open} and the assumption that $|\mathcal{E}_\gamma|\ge n/4$,
\[\mathbb{P}(|\mathcal{C}_0|< +\infty)\le \sum_{n\ge 2}|\Gamma_n|(1-p)^{n/4}\le \sum_{n\ge 2}(n-1)[3(1-p)^{1/4}]^n,\]
which, by Lemma~\ref{lem:threshold} in the Appendix, is strictly less than $1$ for $p> 1-\eta$.
\end{proof}

\section{Proofs of the main results} \label{sec:proofs}

\begin{proof}[Proof of Theorem \ref{thm:arithmetic}]
The proof proceeds by constructing a dependent oriented percolation process on the graph introduced in Section~\ref{sec:regen_perc}. The arithmetic structure of the renewal process yields regeneration at deterministic times, while large infection rates guarantee sufficiently high crossing probabilities between neighbouring blocks.

Let $d>0$ be such that $\mathrm{supp}(\mu)\subseteq d\mathbb{N}$.  
Since $\mu$ is non-degenerate and arithmetic, Lemma~\ref{lem:discrete-U-atom-bound} ensures that
\[
   c  := \sup_{k\in\mathbb{N}} U(\{d \cdot k \}) < 1.
\]
Fix $M\in\mathbb{N}$ such that $c ^M < \eta/2$. Let us write a partition of $\mathbb{Z}$ into blocks of size $M$ as
\[B_k:=\{kM,kM+1, \dots, kM+M-1\}\quad \text{for all }\; k\in\mathbb{Z}\]
and fix the event
\[\mathcal{B}_{k,\ell}:=\left\{\begin{array}{cc}\text{at least one vertex in }B_k\text{ does not have a renewal mark at time }(\ell+1) d
\end{array}\right\}.\]
Then, we obtain for all $k\in\mathbb{Z}$ and $\ell\in\mathbb{N}$,
\begin{equation}\label{eq:renewal.block.event}
\mathbb{P}\left(\mathcal{B}^c_{k,
\ell}\right)\le c^M<\eta/2.
\end{equation}

Define also the events
\[\mathcal{A}^{\nearrow}_{k,\ell}:=\left\{
\begin{array}{c}
    \text{there is an infection path from vertex }kM\text{ to vertex }(k+2)M-1\\
    \text{ within the time interval } \big(\ell\, d, ~(\ell+1)d\big) 
\end{array}
\right\},
\]
\[\mathcal{A}^{\nwarrow}_{k,\ell}:=\left\{
\begin{array}{c}
    \text{there is an infection path from vertex }(k+1)M-1\text{ to vertex }(k-1)M\\
    \text{ within the time interval } \big(\ell\, d, ~(\ell+1)d\big) 
\end{array}
\right\}.
\]

Note that the events $\mathcal{A}^{\nwarrow}_{k,\ell}$ and $\mathcal{A}^{\nearrow}_{k,\ell}$ depend only on the infection marks, while $\mathcal{B}_{k,\ell}$ depends only on the cure marks. Moreover, the event $\mathcal{A}_{k,\ell}^{\nearrow}$ requires that a sequence of $2M-1$ rightward infections occur sequentially (see Figure \ref{fig:infection_M}). Since, sequentially, the time between each infection is an exponential with parameter $\lambda$, it follows, by comparison with Poisson Processes, that
\[\mathbb{P}(\mathcal{A}^{\nearrow}_{k,
\ell})=\mathbb{P}(\textrm{Poisson}(\lambda\, d)\ge 2M-1).\]

\begin{figure}[h]
\centering
\includegraphics[width=15cm]{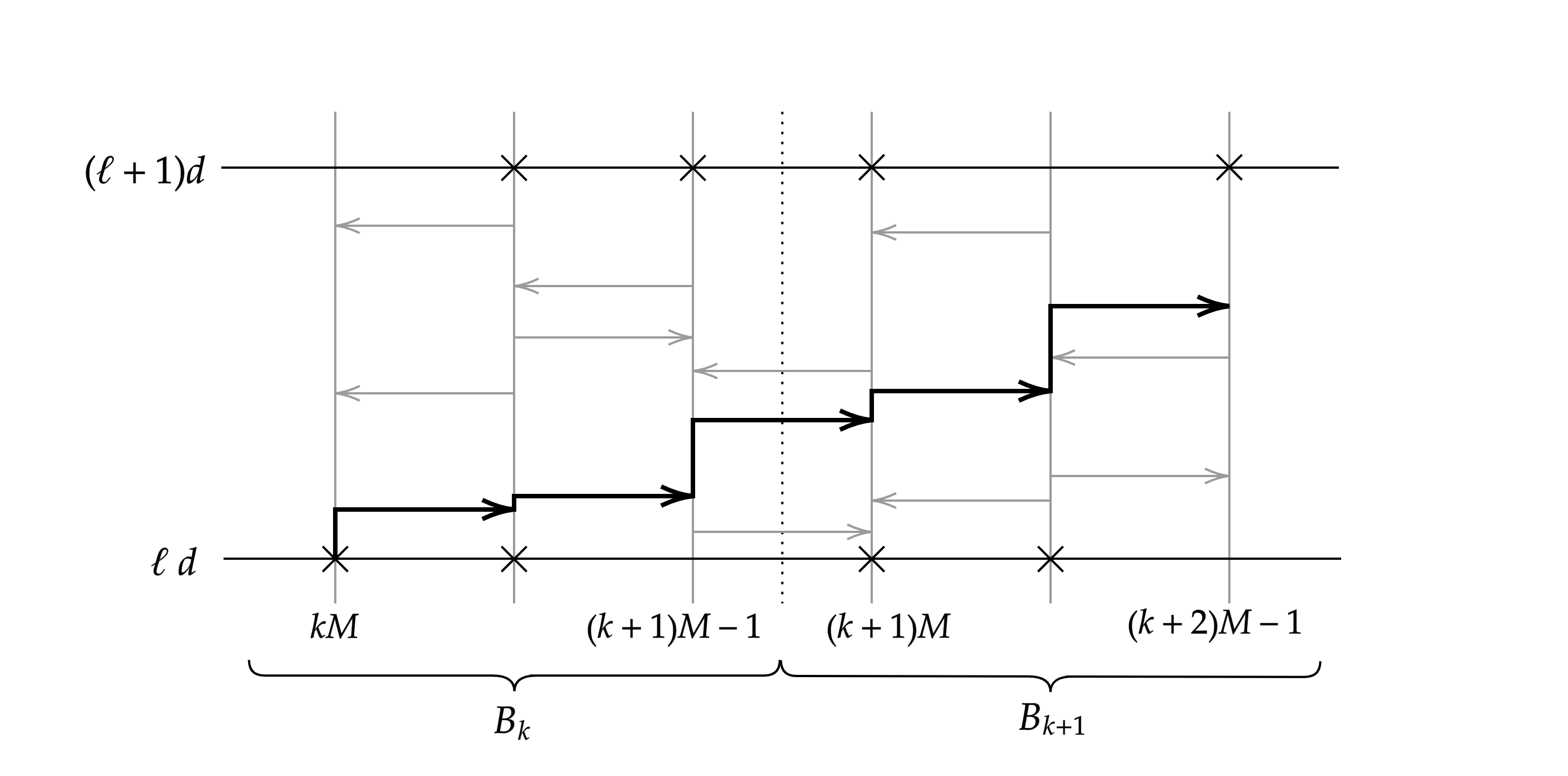}
\caption{Pictorial representation of the event $\mathcal{A}_{k,\ell}^{\nearrow}\cap\mathcal{B}_{k+1,\ell}$ with $M=3$.}\label{fig:infection_M}
\end{figure}

Choose $\lambda = \lambda(M,d,\eta)$ sufficiently large so that for $k,\ell\in\mathbb{N}$, 
\begin{equation}\label{eq:renewal.infection.event}
\mathbb{P}\left((\mathcal{A}^{\nwarrow}_{k,
\ell})^c\right)=\mathbb{P}\left((\mathcal{A}^{\nearrow}_{k,
\ell})^c\right)<\eta/2.
\end{equation}

Consider the set $\mathbb{E}$ as defined in \eqref{eq:def_e} and split
\[\mathbb{E}=\mathbb{E}_\nwarrow \cup \mathbb{E}_\nearrow,\]
where
\[\mathbb{E}_\nwarrow:=\{(v,v')\in\mathbb{V}^2:\text{$v'=v+(-1,1)$}\},\]
\[\mathbb{E}_\nearrow:=\{(v,v')\in\mathbb{V}^2:\text{$v'=v+(1,1)$}\}.\]

Define, for any edge $e=(v,v')\in \mathbb{E}$ with $v=(k,\ell)\in \mathbb{V}$,
\begin{equation}\label{eq:def_d}
\mathcal{D}_{e}:=\begin{cases}
\mathcal{B}_{k-1,\ell}\cap \mathcal{A}_{k,\ell}^\nwarrow,&\text{if $e\in \mathbb{E}_\nwarrow$};\\
\vspace{-.3cm}\\
 \mathcal{B}_{k+1,\ell}\cap \mathcal{A}_{k,\ell}^\nearrow,&\text{if $e\in \mathbb{E}_\nearrow$}.
\end{cases}\end{equation}

Observe that $kM$ is the first vertex of $B_k$, while $(k+2)M - 1$ is the last vertex in $B_{k+1}$ (see Figure~\ref{fig:infection_M}). Moreover, since $\mu$ is arithmetic, there are no cure marks in the interval $(\ell\,d, (\ell+1)d)$. Therefore, on the event $\mathcal{A}_{k,\ell}^{\nearrow}$, there is an infection path from any vertex $v \in B_k$ to any vertex $v' \in B_{k+1}$ within the time interval $(\ell\,d, (\ell+1)d)$. If, in addition, some vertex in $B_k$ is infected at time $\ell\,d$ and $\mathcal{B}_{k+1,\ell}$ occurs, then there will be at least one vertex in $B_{k+1}$ infected at time $(\ell+1)d$.

We consider a bond percolation model on $\overrightarrow{\mathbb{L}}$ as defined in Section~\ref{sec:regen_perc}. In this percolation model, for any $e\in \mathbb{E}$, we take $X_e = 1$ if and only if the event $\mathcal{D}_e$ occurs. Note, by the previous observations, that the percolation event implies survival of the infection in the renewal contact process. Therefore, our task is to prove that $\mathbb{P}(|\mathcal{C}_0|= +\infty)>0$ in this percolation model, which is achieved with the help of Proposition~\ref{prop:percol} and by proving that this percolation model satisfies properties (I) and (II).

Property (II) is easily verified, since $|\mathrm{col}(e_i)-\mathrm{col}(e_j)|\ge 2$ for all $i,j\in\{1,...,k\}$ with $i\neq j$, then $\{\mathcal{D}_{e_i}\}_{i\in\{1,...,k\}}$ are mutually independent as they do not share events regarding the same renewal process.

We now proceed to prove property (I). Note that by \eqref{eq:renewal.block.event}-\eqref{eq:def_d}, for every $e\in \mathbb{E}$, \[\mathbb{P}(X_e=0)=\mathbb{P}(\mathcal{D}^c_e)< \eta.\]

Now consider two edges $e,e' \in \mathbb{E}$ that are colinear with $e \neq e'$. We first consider the case in which $(e,e') \in \mathbb{E}_{\nearrow} \times \mathbb{E}_{\nearrow}$. Since they are colinear, we have that $e = ((k,\ell),(k+1,\ell+1))$ and $e' = ((k,\ell'),(k+1,\ell'+1))$ for some $k,\ell,\ell' \in \mathbb{N}$, where, without loss of generality, we assume $\ell < \ell'$. Then, \eqref{eq:def_d} yields
\[
\mathbb{P}(\mathcal{D}^c_{e}\cap \mathcal{D}^c_{e'})\le\mathbb{P}\Big((\mathcal{A}^{\nearrow}_{k,\ell})^c\cap (\mathcal{A}^{\nearrow}_{k,\ell'})^c\Big)+\mathbb{P}\Big(\mathcal{B}_{k+1,\ell}^c\cap (\mathcal{A}^{\nearrow}_{k,\ell'})^c\Big)+\mathbb{P}\Big((\mathcal{A}^{\nearrow}_{k,\ell})^c\cap \mathcal{B}_{k+1,\ell'}^c\Big)+\mathbb{P}\Big(\mathcal{B}^c_{k+1,\ell}\cap \mathcal{B}^c_{k+1,\ell'}\Big).\] 

The sum above has 4 terms. We shall prove that each of them is has an upper bound of $\eta^2/4$ and, as a result, $\mathbb{P}(\mathcal{D}^c_{e}\cap \mathcal{D}^c_{e'})< \eta^2$.

Conditioned on $\mathcal{B}^c_{k+1,\ell}$, all vertices in $B_{k+1}$ have cure marks at time $(\ell+1)d$. By the regenerative property of renewal processes applied simultaneously to all these vertices, we have that $\mathbb{P}(\mathcal{B}^c_{k+1,\ell'} \mid \mathcal{B}^c_{k+1,\ell}) = \mathbb{P}(\mathcal{B}^c_{k+1,\ell' - \ell})$. Therefore,

\[\mathbb{P}(\mathcal{B}^c_{k+1,\ell}\cap\mathcal{B}^c_{k+1,\ell'})=\mathbb{P}(\mathcal{B}^c_{k+1,\ell})\mathbb{P}(\mathcal{B}^c_{k+1,\ell'-\ell})< \eta^2/4.
\]
By independence between events, it is easy to show that all the other terms of the sum are also upper bounded by $\eta^2/4$. Hence, $\mathbb{P}(\mathcal{D}_{e} \cap \mathcal{D}_{e'}) < \eta^2$ whenever $(e,e') \in \mathbb{E}_{\nearrow} \times \mathbb{E}_{\nearrow}$ are colinear with $e \neq e'$.

Observe that analogous arguments can be used for the cases $(e,e') \in \mathbb{E}_{\nearrow} \times \mathbb{E}_{\nwarrow}$, $(e,e') \in \mathbb{E}_{\nwarrow} \times \mathbb{E}_{\nearrow}$, or $(e,e') \in \mathbb{E}_{\nwarrow} \times \mathbb{E}_{\nwarrow}$. Note also that analogous arguments apply to any number of distinct colinear edges, yielding
\[
\mathbb{P}(\mathcal{D}^c_{e_1},\ldots,\mathcal{D}^c_{e_k}) < \eta^k, \qquad \text{for any distinct colinear } e_1,\ldots,e_k \in \mathbb{E}.
\]

Therefore, property (I) holds and we can apply Proposition~\ref{prop:percol} to show that the renewal contact process survives with positive probability for the chosen $\lambda$. This implies $\lambda_c<+\infty$.

\end{proof}

\begin{proof}[Proof of Theorem \ref{thm:bounded}]

The proof follows the same comparison scheme as in Theorem~\ref{thm:arithmetic}. The local control of the renewal measure provided by Proposition~\ref{prop:uniform-local-U_at} replaces the deterministic regeneration times available in the arithmetic setting. %The proof follows by arguments similar to those used in the proof of Theorem \ref{thm:arithmetic}, namely, by comparison with the percolation models defined in Section~\ref{sec:regen_perc}. 

By the combination of the hypothesis and Proposition~\ref{prop:uniform-local-U_at}, there exists $\nu>0$ such that
\[
   u_\nu:=\sup_{a\ge 0} U\big( (a,a+\nu] \big) \;<\; \eta/2.
\]
Fix such a $\nu$. For all $j\in\mathbb{Z}$ and $\ell\in \mathbb{N}_0$, define
\[
\mathcal{B}'_{j,\ell}=\{\text{there is no recovery mark at vertex $j$ during the time interval $\big(\ell\, \nu, ~(\ell+1)\nu]$}\}
\]
and note that, by Markov inequality,
\begin{equation}\label{eq:p_b}
\begin{aligned}
\mathbb{P}\big((\mathcal{B}'_{j,\ell})^c\big)&\le\mathbb{E}(\text{number of recovery marks at vertex $j$ during the time interval $\big(\ell\, \nu, ~(\ell+1)\nu]$})\\
&=U\big((\ell\, \nu, ~(\ell+1)\nu]\big)\\
&\le u_\nu.
\end{aligned}
\end{equation}

Define the events
\[
\mathcal{A}'^{\nearrow}_{j,\ell}:=\left\{
\begin{array}{c}
    \text{there is an attempted rightward infection from vertex }j\text{ to vertex }j+1\\
    \text{ within the time interval } \big(\ell\, \nu, ~(\ell+1)\nu\big] 
\end{array}
\right\},
\]
\[
\mathcal{A}'^{\nwarrow}_{j,\ell}:=\left\{
\begin{array}{c}
    \text{there is an attempted leftward infection from vertex }j\text{ to vertex }j-1\\
    \text{ within the time interval } \big(\ell\, \nu, ~(\ell+1)\nu\big]  
\end{array}
\right\}.
\]

It is easy to see that 
\[
\mathbb{P}(\mathcal{A}'^{\nearrow}_{j,\ell})=\mathbb{P}(\mathcal{A}'^{\nwarrow}_{j,\ell})=\mathbb{P}(\textrm{Poisson}(\lambda\, \nu)\ge 1).
\]
Choose $\lambda = \lambda(u_\nu,\nu,\eta)$ sufficiently large such that
\begin{equation}\label{eq:p_a}
\mathbb{P}\big((\mathcal{A}'^\nearrow_{j,\ell})^c\big)< \eta -2u_\nu.
\end{equation}

Consider
\[
\mathbb{E}:=\mathbb{E}_\nwarrow \cup \mathbb{E}_\nearrow,
\]
where
\[
\mathbb{E}_\nwarrow:=\{(v,v')\in\mathbb{V}^2:\text{$v'=v+(-1,1)$}\},
\]
\[
\mathbb{E}_\nearrow:=\{(v,v')\in\mathbb{V}^2:\text{$v'=v+(1,1)$}\}.
\]

Define, for any edge $e=(v,v')\in \mathbb{E}$ with $v=(k,\ell)\in \mathbb{V}$,
\begin{equation}\label{eq:def_d'}
\mathcal{D}'_{e}:=\begin{cases}
\mathcal{B}'_{k-1,\ell}\cap \mathcal{B}'_{k,\ell}\cap \mathcal{A}_{k,\ell}'^\nwarrow,&\text{if $e\in \mathbb{E}_\nwarrow$};\\
\vspace{-.3cm}\\
 \mathcal{B}'_{k,\ell}\cap \mathcal{B}'_{k+1,\ell}\cap \mathcal{A}_{k,\ell}'^\nearrow,&\text{if $e\in \mathbb{E}_\nearrow$},
\end{cases}
\end{equation}
the event that, during the interval $(\ell\, \nu,(\ell+1)\nu]$, there is an attempted infection from vertex $k$ to $k-1$ (or $k+1$) and there are no recovery marks at either $k$ or $k-1$ (or $k+1$). In particular, if vertex $k$ is infected  immediately after time $\ell\, \nu$ and $\mathcal{D}'_e$ occurs, then vertex $k-1$ (or $k+1$) will be infected right after time $(\ell +1)\nu$.

As in the proof of Theorem \ref{thm:arithmetic}, we consider a bond percolation model on $\overrightarrow{\mathbb{L}}$ where, for any $e\in \mathbb{E}$, we take $X_e = 1$ if and only if the event $\mathcal{D}'_e$ occurs. Note that the percolation event implies survival of the infection in the renewal contact process. Therefore, the proof is complete by Proposition~\ref{prop:percol} once we show that this percolation model satisfies properties (I) and (II).

As in Theorem \ref{thm:arithmetic}, property (II) is easily verified, and we now proceed to prove property (I). Observe by \eqref{eq:p_b}--\eqref{eq:def_d'}, for every $e\in \mathbb{E}$,
\[
\mathbb{P}(X_e=0)=\mathbb{P}\big((\mathcal{D}'_e)^c\big)< \eta.
\]

Now consider two edges $e,e' \in \mathbb{E}$ that are colinear with $e \neq e'$. We first consider the case in which $(e,e') \in \mathbb{E}_{\nearrow} \times \mathbb{E}_{\nearrow}$. Since they are colinear, we have $e = ((k,\ell),(k+1,\ell+1))$ and $e' = ((k,\ell'),(k+1,\ell'+1))$ for some $k,\ell,\ell' \in \mathbb{N}$, where, without loss of generality, we assume $\ell < \ell'$. Then, \eqref{eq:def_d'} yields
\begin{equation}\label{eq:inter_d'}
\begin{aligned}
\mathbb{P}((\mathcal{D}'_{e})^c\cap (\mathcal{D}'_{e'})^c\big)\le&\mathbb{P}\Big((\mathcal{A}'^{\nearrow}_{j,\ell})^c\cap (\mathcal{A}'^{\nearrow}_{j,\ell'})^c\Big)+\sum_{i=j}^{j+1}\mathbb{P}\Big((\mathcal{B}'_{i,\ell})^c\cap (\mathcal{A}'^{\nearrow}_{j,\ell'})^c\Big)\\
&+\sum_{i=j}^{j+1}\mathbb{P}\Big((\mathcal{A}'^{\nearrow}_{j,\ell})^c\cap (\mathcal{B}'_{i,\ell'})^c\Big)
+\sum_{i=j}^{j+1}\sum_{i'=j}^{j+1}\mathbb{P}\Big((\mathcal{B}'_{i,\ell})^c\cap (\mathcal{B}'_{i',\ell'})^c\Big).
\end{aligned}
\end{equation}

It is possible that there are some recovery marks at vertex $i$ during the interval $(\ell\, \nu,(\ell+1)\nu]$. Define $Z_{i,\ell}$ as the time of the first such mark, with $Z_{i,\ell}:=-\infty$ if there is no recovery mark at vertex $i$ during this time interval. Conditioning on $Z_{i,\ell}=t\in (\ell\, \nu,(\ell+1)\nu]$ and using the regenerative property, we may consider that the subsequent recovery marks follow a new independent renewal process starting from $t$. Therefore, analogously to (\ref{eq:p_b}), we have for $t\in (\ell\, \nu,(\ell+1)\nu]$,
\[
\begin{aligned}
\mathbb{P}\big((\mathcal{B}'_{i,\ell'})^c \mid Z_{i,\ell}=t\big)
&= \mathbb{P}\begin{pmatrix}
\text{there is at least one recovery mark at vertex $i$ during}\\
\text{the time interval $\big(\ell' \nu-t,(\ell'+1)\nu-t\big]$}
\end{pmatrix}\\
&\le u_\nu.
\end{aligned}
\]

Conditioned on $(\mathcal{B}'_{i,\ell})^c$, we have $Z_{i,\ell}=t$ for some $t\in(\ell \nu,(\ell+1)\nu]$. Therefore,
\[
\mathbb{P}\big((\mathcal{B}'_{i,\ell'})^c \mid (\mathcal{B}'_{i,\ell})^c\big)\le u_\nu
\]
and
\[
\mathbb{P}\big((\mathcal{B}'_{i,\ell})^c\cap(\mathcal{B}'_{i,\ell'})^c\big)
=\mathbb{P}\big((\mathcal{B}'_{i,\ell})^c\big)\mathbb{P}\big((\mathcal{B}'_{i,\ell'})^c \mid (\mathcal{B}'_{i,\ell})^c\big)
\le u^2_\nu.
\]

The probabilities of every other intersection in \eqref{eq:inter_d'} are easily bounded using independence of events together with \eqref{eq:p_b} and \eqref{eq:p_a}. Thus, we obtain
\[
\mathbb{P}\big((\mathcal{D}'_{e})^c\cap (\mathcal{D}'_{e'})^c\big)< (\eta-2u_\nu)^2+4(\eta-2u_\nu)u_\nu+4u_\nu^2=\eta^2.
\]

Hence,
\(
\mathbb{P}((\mathcal{D}'_{e})^c\cap (\mathcal{D}'_{e'})^c) < \eta^2
\)
whenever $(e,e') \in \mathbb{E}_{\nearrow} \times \mathbb{E}_{\nearrow}$ are colinear with $e \neq e'$. Note that analogous arguments apply to the cases $(e,e') \in \mathbb{E}_{\nearrow} \times \mathbb{E}_{\nwarrow}$, $(e,e') \in \mathbb{E}_{\nwarrow} \times \mathbb{E}_{\nearrow}$, or $(e,e') \in \mathbb{E}_{\nwarrow} \times \mathbb{E}_{\nwarrow}$. 

Now, consider $k>2$ and $(e_1,...,e_k)\in \mathbb{E}_{\nearrow}^k$ colinear. Then we may write $e_i=(j,\ell_i)$, $i=1,...,k$, for some $j$ and a sequence of distinct $\ell_1,...,\ell_k$, which we denote generically by $\ell$ when the index is not important. Similarly to \eqref{eq:inter_d'}, we can decompose 
\(
\mathbb{P}\big(\bigcap_{i=1}^k (\mathcal{D}'_{e_i})^c\big)
\)
into a sum of several terms, where the probability of the intersection of $a$ events of type $(\mathcal{A}'^{\nearrow}_{j,\ell})^c$, $b$ events of type $(\mathcal{B}'_{j,\ell})^c$, and $c$ events of type $(\mathcal{B}'_{j+1,\ell})^c$ appears ${k \choose a,b,c}$ times. Therefore, by the trinomial expansion and using arguments analogous to the case $k=2$,
\[
\mathbb{P}\big((\mathcal{D}'_{e_1})^c,\ldots,(\mathcal{D}'_{e_k})^c\big) 
< \sum_{\substack{a,b,c\, \in\{0,1,..,.k\}\\ a+b+c=k}} {k\choose a,b,c}(\eta-2u_\nu)^{a}(u_\nu)^b(u_\nu)^{c}
=\eta^k.
\]

It is easy to see that an analogous argument extends from $(e_1,...,e_k)\in \mathbb{E}_{\nearrow}^k$ to any $(e_1,...,e_k)\in \mathbb{E}^k$. Therefore, property (I) holds, and we can apply Proposition~\ref{prop:percol} to conclude that the renewal contact process survives with positive probability for the chosen $\lambda$. This implies $\lambda_c<+\infty$.
\end{proof}

\section*{Acknowledgments}
Research supported by grants \#2023/13453-5, \#2024/06021-4 and \#2025/27064-6, S\~ao Paulo Research Foundation (FAPESP). This study was financed in part by the Coordena\c{c}\~ao
de Aperfei\c{c}oamento de Pessoal de N\'{\i}vel Superior - Brasil (CAPES) - Finance Code 001. It was also supported by the National Institute of Science and Technology in Stochastic Modeling and Complexity (INCT-NUMEC), funded by CNPq (grant no. 408590/2024-6).

The authors used a Large Language Models from DeepSeek and OpenAI to improve the clarity and linguistic flow of the manuscript. All mathematical content and conclusions remain the sole responsibility of the authors.

\appendix

\section{Contour counting estimates} \label{sec:appendix}

In this appendix we derive the contour-counting estimate used in the proof of Proposition~\ref{prop:percol}, and record a convenient explicit threshold for the associated power series.

We work with the first quadrant of the square lattice $\mathbb Z_{\ge0}^2$. Equivalently this model is isomorphic to dual graph $\smash{\stackrel{\rightarrow}{\mathbb{L}}}
_*$; however, for clarity we state the path model directly on $\mathbb Z_{\ge0}^2$.

Fix an integer $n\ge2$. An admissible path of length $n$ $\pi=(v_0,v_1,\dots,v_n)$ is a self-avoiding lattice path satisfying:
\begin{enumerate}[topsep=2pt,itemsep=2pt]
  \item $v_0=(0,k)$ for some integer $k$ with $1\le k\le n-1$;
  \item the first edge is: $v_1=(1,k)$ (the step $(1,0)$);
  \item for every intermediate vertex $1\le i\le n-1$ we have $y(v_i)\ge1$ (so the path meets the line $y=0$ only possibly at the final vertex);
  \item the last edge is downward: for some $\ell\ge1$ we have $v_{n-1}=(\ell,1)$ and $v_n=(\ell,0)$ (the step $(0,-1)$);
  \item the path is self-avoiding: $v_i\ne v_j$ for $i\ne j$.
\end{enumerate}
Let $c_n$ be the number of such admissible paths of length $n$. We give a short rigorous counting argument that produces a simple upper bound valid for every $n$.

\begin{proposition}\label{prop:peierls}
For every integer $n\ge2$,
\[
c_n \le (n-1)\,3^{\,n-2}.
\]
\end{proposition}

\begin{proof}
Fix $n\ge2$. For each admissible path $\pi=(v_0,\dots,v_n)$ the first edge $v_0\to v_1$ is fixed (it is $(1,0)$), and the last edge $v_{n-1}\to v_n$ is fixed (it is $(0,-1)$). Thus the path is determined by the sequence of the $n-2$ interior moves between $v_1$ and $v_{n-1}$.

We obtain an upper bound by relaxing the self-avoidance condition: at each interior step (when we are at a vertex $v_i$ with $1\le i\le n-2$) we forbid only the immediate backtrack to the previous vertex $v_{i-1}$, but we do not forbid revisiting other earlier vertices. Under this relaxation each interior step has at most three possible choices (the four neighbors minus the immediate predecessor). Therefore the number of possible interior move sequences of length $n-2$ under this relaxed rule is at most $3^{\,n-2}$.

The relaxed count upper-bounds the true count of self-avoiding interior sequences, because any self-avoiding continuation also satisfies the no-immediate-backtrack condition. Consequently, for each fixed starting height $k$ there are at most $3^{\,n-2}$ admissible interior sequences, hence at most $3^{\,n-2}$ admissible full paths with that $k$. Finally, there are at most $n-1$ choices for $k\in\{1,\dots,n-1\}$. Summing over
$k$ yields
\[
c_n \le (n-1)\,3^{\,n-2},
\]
which proves the proposition.
\end{proof}

The inequality of Proposition~\ref{prop:peierls} gives a fully explicit uniform bound which is convenient for deriving absolute-convergence criteria
for auxiliary power series. The next lemma produces a rational threshold that is easy to cite in the main text.

\begin{lemma}\label{lem:threshold}
Let $\varepsilon \in \left(0, \ {1}/{2^8}\right)$, then
\[
S(\varepsilon):=\sum_{n\ge2}(n-1)3^{\,n-2}\,\varepsilon^{\,n/4}<1.
\]
\end{lemma}

\begin{proof}
Set $r:=\varepsilon^{1/4}$ (so $r\in(0,1)$) and $s:=3r$. For $|s|<1$
the identity $\sum_{k\ge1} k s^{k}=s/(1-s)^2$ holds and hence
\[
\sum_{n\ge2}(n-1)s^n=\frac{s^2}{(1-s)^2}.
\]
Therefore, for $3r<1$,
\[
S(\varepsilon)=\sum_{n\ge2}(n-1)3^{n-2}r^{n}
=\frac{1}{9}\sum_{n\ge2}(n-1)s^{n}
=\frac{r^2}{(1-3r)^2}.
\]
Assume now that $r\le 1/4$ (equivalently $\varepsilon\le 2^{-8}$). Then $3r\le 3/4$ and therefore $1-3r\ge 1/4$. Thus
\[
S(\varepsilon)=\frac{r^2}{(1-3r)^2}\le\frac{(1/4)^2}{(1/4)^2}=1,
\]
with strict inequality if $r<1/4$ (i.e. $\varepsilon<2^{-8}$). This proves the lemma.
\end{proof}

\medskip

\bibliography{references}
\end{document}